\documentclass{amsart}

\usepackage{lineno,hyperref}
\usepackage[english]{babel}
\usepackage[utf8]{inputenc}
\usepackage{amsmath}
\usepackage{easybmat}
\newcommand\numberthis{\addtocounter{equation}{1}\tag{\theequation}}
\usepackage{graphicx}
\usepackage[colorinlistoftodos]{todonotes}
\usepackage{mathrsfs}
\usepackage{yfonts}
\usepackage{eufrak}
\usepackage{bbm}
\usepackage{amssymb}                 
\usepackage{amsfonts}
\usepackage{amsmath}
\usepackage{euscript}
\usepackage{amsthm}
\usepackage{yhmath}
\usepackage{epsfig}
\usepackage{environ}
\usepackage{caption}
\usepackage{indentfirst}
\usepackage[usenames,dvipsnames]{pstricks}
\usepackage{pst-grad}
\usepackage{pst-plot}
\usepackage{tikz}
\usepackage{mathtools}

\DeclarePairedDelimiter\floor{\lfloor}{\rfloor}
\usetikzlibrary{quotes,angles,positioning}
\usetikzlibrary{decorations.shapes}
\usetikzlibrary{calc}
\usetikzlibrary{arrows.meta}

\chardef\bslash=`\\ 

\newtheorem{thm}{Theorem}[section]

\newtheorem*{theorem_A}{Theorem A}

\newtheorem*{thm_B}{Theorem B}
\newtheorem{cor}[thm]{Corollary}
\newtheorem{lem}[thm]{Lemma}
\newtheorem{prop}[thm]{Proposition}

\newtheorem{example}{Example}
\newtheorem*{exmpl}{Example}
\newtheorem*{prop_spc}{Proposition 3.3}
\newtheorem*{prop_spc2}{Proposition 4.1}
\theoremstyle{definition}
\newtheorem{defn}{Definition}[section]
\newtheorem{rem}{Remark}[section]
\newtheorem{notation}{Notation}
\newtheorem*{observation}{Observation}
\theoremstyle{remark}

\DeclareMathOperator{\aut}{\mathrm{Aut}}
\DeclareMathOperator{\Hom}{Hom}
\DeclareMathOperator{\gl}{GL}
\DeclareMathOperator{\pgl}{PGL}

\DeclareMathOperator{\stab}{Stab}

\DeclareMathOperator{\rk}{\mathrm{rk}}

\DeclareMathOperator{\End}{End}

\DeclareMathOperator{\chc}{\mathrm{char}}
\DeclareMathOperator{\diag}{diag}

\DeclareMathOperator{\snf}{\mathcal{SNF}}
\DeclareMathOperator{\plu}{Pl\ddot{u}}

\makeatletter
\newcommand{\extp}{\@ifnextchar^\@extp{\@extp^{\,}}}
\def\@extp^#1{\mathop{\bigwedge\nolimits^{\!#1}}}
\makeatother

\NewEnviron{myequation}{%
\begin{equation*}
\scalebox{0.8}{$\BODY$}
\end{equation*}
}
\modulolinenumbers[5]
\usetikzlibrary{patterns}

\newlength{\hatchspread}
\newlength{\hatchthickness}
\newlength{\hatchshift}
\newcommand{\hatchcolor}{}
\tikzset{hatchspread/.code={\setlength{\hatchspread}{#1}},
         hatchthickness/.code={\setlength{\hatchthickness}{#1}},
         hatchshift/.code={\setlength{\hatchshift}{#1}},
         hatchcolor/.code={\renewcommand{\hatchcolor}{#1}}}
\tikzset{hatchspread=3pt,
         hatchthickness=0.4pt,
         hatchshift=0pt,
         hatchcolor=black}
\pgfdeclarepatternformonly[\hatchspread,\hatchthickness,\hatchshift,\hatchcolor]
   {custom north west lines}
   {\pgfqpoint{\dimexpr-2\hatchthickness}{\dimexpr-2\hatchthickness}}
   {\pgfqpoint{\dimexpr\hatchspread+2\hatchthickness}{\dimexpr\hatchspread+2\hatchthickness}}
   {\pgfqpoint{\dimexpr\hatchspread}{\dimexpr\hatchspread}}
   {
    \pgfsetlinewidth{\hatchthickness}
    \pgfpathmoveto{\pgfqpoint{0pt}{\dimexpr\hatchspread+\hatchshift}}
    \pgfpathlineto{\pgfqpoint{\dimexpr\hatchspread+0.15pt+\hatchshift}{-0.15pt}}
    \ifdim \hatchshift > 0pt
      \pgfpathmoveto{\pgfqpoint{0pt}{\hatchshift}}
      \pgfpathlineto{\pgfqpoint{\dimexpr0.15pt+\hatchshift}{-0.15pt}}
    \fi
    \pgfsetstrokecolor{\hatchcolor}
    \pgfusepath{stroke}
   }

\usepackage[a4paper,top=3cm,bottom=2cm,left=3cm,right=3cm,marginparwidth=1.75cm]{geometry}

\usepackage{amsmath}
\usepackage{graphicx}
\usetikzlibrary{patterns}
\usepackage[colorinlistoftodos]{todonotes}

\newcommand{\eval}[2][\right]{\relax
  \ifx#1\right\relax \left.\fi#2#1\rvert}

\let\abs=\envert

\begin{document}
\title[Quantum polynomials]{Automorphisms of quantum polynomials}
  
\author[A. Gupta]{Ashish Gupta}
\address{Department of Mathematics\\
Ramakrishna Mission Vivekananda Educational and Research Institute\\
Belur, \\
Howrah, WB 711202\\
India}
\email{a0gupt@gmail.com\thanks{}} 

\begin{abstract}
An important step in the determination of the automorphism group of the quantum torus of rank $n$ (or twisted group algebra of $\mathbb Z^n$) is the determination of its so-called non-scalar automorphisms.  
We present a new algorithimic approach towards this problem based on the bivector representation $\extp^2: \gl(n, Z) \rightarrow \gl(\binom{n}{2}, \mathbb Z)$ of $\gl(n, \mathbb Z)$ and thus compute the non-scalar automorphism group $\aut(\mathbb Z^n, \lambda)$ in several new cases. As an application of our ideas we show that the  quantum polynomial algebra (multiparameter quantum affine space of rank $n$) has only scalar (or toric) automorphisms provided that the torsion-free rank of the subgroup generated by the defining multiparameters is no less than $\binom{n - 1}{2} + 1$ thus improving an earlier result. We also investigate the question: when is a multiparameter quantum affine space free of so-called linear automorphisms other than those arising from the action of the $n$-torus ${(\mathbb F^\ast)}^n$.
\end{abstract}
\maketitle
\tableofcontents
\section{Introduction}\label{intro}

Quantum polynomials are non-commutative versions of polynomial and Laurent polynomial algebras. The difference with the ordinary polynomials lies in the fact that the commutativity of the variables  is replaced by quasi-commutativity, that is, $X_iX_j = q_{ij}X_jX_i$ for non-zero scalars $q_{ij}$. This relation is the Weyl form of the cannonical commutation relation of quantum mechanics.
For the Laurent case (where each variable $X_j$ has an inverse) there other terms in use, for example, quantum torus, twisted group algebra, McConnell--Pettit algebra etc.

Quantum polynomials play a key role in the theory of quantum groups~\cite{ART1997} and also in non-commutative geometry~\cite{Man}. Their Laurent versions arise in Lie theory as coordinate structures of extended affine Lie algebras~\cite{NeebKH2008} and also in the representation theory of nilpotent groups~\cite{Br2000}. The question of automorphisms of polynomial algebras is a still open one (the Jacobian conjecture) and it is therefore of interest to investigate the automorphisms of quantum polynomial algebras. 
It is generally believed that these algebras have relatively fewer automorphisms, that is, they are rigid. However, the complete understanding of their automorphisms groups is yet to be had and the results in this paper are expected to be a step in this direction.

We briefly recall the definitions. A quantum affine space $\mathcal O_{\mathfrak q}(\mathbb F^n)$ over a field $\mathbb F$ is defined as  

\begin{align}
\label{quantum_space_def}
\mathcal O_{\mathfrak q}(\mathbb F^n) &:= \mathbb F\langle X_1, X_2, \cdots, X_n \rangle/(X_i X_j - q_{ij}X_j X_i)
\end{align}

where $\mathfrak q$ stands for the matrix of \emph{multiparameters} $q_{ij}$, that is, 
$\mathfrak q  =  (q_{ij})$ and \[ q_{ij} \in \mathbb F^\ast := \mathbb F \smallsetminus \{0\}. \]   
The matrix $\mathfrak{q}$ is assumed to be multiplicatively anti-symmetric, that is, 
$q_{ii} = 1$ and $q_{ji} = q_{ij}^{-1}.$ 
A given quantum affine space $\mathcal O_{\mathfrak q}(\mathbb F^n)$ can be embedded in a \emph{quantum torus} $\mathcal O_{\mathfrak q}({(\mathbb F^\ast)}^n)$ by means of localization. The latter type of algebra is generated by the indeterminates $X_1, \cdots, X_n$ together with their inverses subject to the quantum commutation relations as in (\ref{quantum_space_def}).  The following notion plays an important  role in the theory of quantum polynomials.

\begin{defn}
For the quantum polynomial algebras $\mathcal O_{\mathfrak q}({\mathbb F}^n)$ and
$\mathcal O_{\mathfrak q}({(\mathbb F^\ast)}^n)$ the $\lambda$-group (denoted $\Lambda$) 
is the subgroup of $\mathbb F^\ast$ generated by the multiparameters $q_{ij}$. 
\end{defn}

For the quantum tori the monomials 
$\mathbf X^{\mathbf m} := X_1^{m_1}\cdots X_n^{m_n}$ are units and the group-theoretic commutator $[\mathbf X^{\mathbf m}, \mathbf X^ {\mathbf m'}]$ defined as 
\[ [\mathbf X^{\mathbf m}, \mathbf X^ {\mathbf m'}] := \mathbf X^{\mathbf m} \mathbf X^{\mathbf m'}{(\mathbf X^{\mathbf m})}^{-1}{(\mathbf X^{\mathbf m'})}^{-1} \] 
yields an alternating $\mathbb Z$-bilinear function  \begin{equation}
\label{lmbdadefn}    
\lambda: \Gamma  \times \Gamma  \rightarrow \mathbb F^\ast,  \ \ \ \ \ \ \ \ \ \   \lambda(\mathbf m, \mathbf m') = [\mathbf X^{\mathbf m},\mathbf X^{\mathbf m'}],  \ \ \ \ \ \ \ \ \ \ \forall  \mathbf m, \mathbf m' \in \Gamma := \mathbb Z^n.  
\end{equation} 
whose image is contained in the group $\Lambda$ (e.g., \cite[Section 1]{OP1995}).
Let   
  \begin{equation}\label{dcpms-Lmbda}
      \Lambda : =  \langle p_1 \rangle \times \langle p_2 \rangle \times \cdots \times \langle p_l \rangle, \ \ \ \ \ \ \ \ \  l \in \mathbb N. \end{equation}
It was observed in  \cite[Lemma 3.3(ii)]{OP1995} (see Section \ref{aut-qtor}) that in the study of the autmorphisms of a quantum torus the crucial case is where the group $\Lambda$ is torsion-free.
We may thus assume that each direct summand $\langle p_i \rangle\ (i = 1, \cdots, l)$ in \eqref{dcpms-Lmbda} is an infinite cycle.   
For \[ \mathbf m, \mathbf m' \in \Gamma = \mathbb Z^n \] we thus have 
\begin{equation}\label{e_i-forms}
     \lambda(\mathbf m, \mathbf m') = p_1^{e_1(\mathbf m, \mathbf m')}p_2^{e_2(\mathbf m, \mathbf m')}\cdots p_l^{e_l(\mathbf m, \mathbf m')}
\end{equation}
 where each exponent map $e_i: \Gamma \times \Gamma \rightarrow \mathbb Z$ yields an alternating bilinear form on $\Gamma$.
 
 It is known (e.g., \cite{MP}) that the units of a quantum torus algebra are trivial, that is, are of the form $\alpha \mathbf X^{\mathbf m}$, where $\alpha \in \mathbb F^\ast$. 
 Let us denote by $\mathscr A$ the group $\aut_{\mathbb F}(\mathcal O_{\mathfrak q}({(F^\ast)}^n))$ of all 
 $\mathbb F$-automorphisms and by $\mathscr U$ the group of trivial units of the algebra $\mathcal O_{\mathfrak q}({(F^\ast)}^n)$.  
It is easily seen (e.g., \cite{OP1995}) that the action of the group $\mathscr A$ on the quantum torus $\mathcal O_{\mathfrak q}({(F^\ast)}^n)$ induces an action of this same group on the group $\mathscr U$ of trivial units fixing $\mathbb F^\ast$ elementwise.  There is thus an action of the group $\mathscr A$ on the quotient group $\mathscr U/\mathbb F^\ast \cong \Gamma$ yielding a homomorphism 
\begin{equation}\label{actn-Gma} \mathscr A \longrightarrow \aut \Gamma = \gl(n, \mathbb Z) \end{equation} 
whose kernel is the group $\mathscr S$ 
of all \emph{ scalar automorphisms} defined by $\psi(\mathbf X^{\mathbf m}) = \phi(\mathbf m)(\mathbf X^{\mathbf m})$ for $\phi \in  \Hom(\Gamma, \mathbb F^\ast)$~\cite{OP1995}. Thus  $\mathscr S \cong (\mathbb F^\ast)^n$.
 Furthermore by \cite[Lemma 3.3(iii)]{OP1995} the image of the map in \eqref{actn-Gma} 
coincides with the group $\aut(\mathbb Z^n, \lambda) \le \gl(n, \mathbb Z)$ of  all \emph{non-scalar automorphisms} $\sigma$ of $\Gamma$ satisfying  
\begin{equation}
\label{form_prsvng}
\lambda (\sigma \mathbf {m}, \sigma \mathbf {m'}) = \lambda(\mathbf {m}, \mathbf {m'} ) \ \ \ \ \ \ \ \  \forall \bf{m}, \bf{m'} \in \mathbb Z^n.    
\end{equation}
We thus have the following exact sequence for the group $\mathscr A$ as noted in \cite{NeebKH2008}
\begin{equation}\label{Neeb-Key-Result1} 
1 \rightarrow \mathscr S \rightarrow \mathscr A  \rightarrow \aut(\mathbb Z^n, \lambda) \rightarrow 1. \end{equation}

In terms of alternating bilinear forms $e_i$ defined above the following characterization of the non-scalar automorphism group $\aut(\mathbb Z^n, \lambda)$ (as noted in \cite{OP1995}) is immediate.  
\begin{thm}[Theorem 3.4 of \cite{OP1995}] 
 \label{OP-charac}
 Let $\mathrm{Sp}(\mathbb Z, e_i)$ be the symplectic group associated with the form $e_i$. Then
 \begin{equation}\label{NsAut-OP-frmla}
 \aut(\mathbb Z^n, \lambda) = \bigcap_{i = 1}^l \mathrm{Sp}(\mathbb Z, e_i).
 \end{equation}
 \end{thm}

 Thus a non-scalar automorphism must preserve each of the forms $e_i$ and so must stabilize the radicals of each of these forms. This fact has been fruitfully used in \cite{OP1995} towards determining the non-scalar automorphism group $\aut(\mathbb Z^n, \lambda)$ in certain cases. 
However, notwithstanding the nicety of the preceding theorem from a computational viewpoint \eqref{NsAut-OP-frmla} seems to entail some difficulties limiting the cases where it can be directly applied.

In the following we present a new approach towards the computation of this group. 
Before stating it we fix some notation.

\begin{notation}\label{not1}
In view of $\eqref{dcpms-Lmbda}$ let  
\begin{equation} \label{rel-matrx}
    \lambda(\mathbf e_i, \mathbf e_j) = p_1^{ m_{1}^{(ij)}} \cdots  p_l^{ m_{l}^{(ij)} }, \ \ \ \ \ \ \ \  \ \ \ \  \forall  1 \le i < j \le n,  
\end{equation} 
where $\mathbf e_i, \mathbf e_j$ are standard basis vectors of the free $\mathbb Z$-module $\Gamma$. On the $\binom{n}{2}$ pairs $(ij),\ (i < j)$ we assume the lexicographic order. Let $\mathsf M \in \mathrm{Mat}_{\binom{n}{2} \times l}(\mathbb Z)$ be the 
matrix whose $((ij), s)$ entry is the exponent $m_{s}^{(ij)}$ of $p_s$ in \eqref{rel-matrx} ($\forall s \in \{1. \cdots, l\}$). 
\end{notation}

With this notation we have the following.
\begin{theorem_A}
For a quantum torus $\mathcal O_{\mathfrak q} (\mathbb (F^\ast)^n )$ suppose that the group $\Lambda$ generated by the multiparameters $q_{ij}$ is torsion-free. Set $N  =  \binom{n}{2}$ and let $\mathsf M$ be as defined in Notation \ref{not1} above.  Then 
\[  \aut (\mathbb Z^n, \lambda) = \bigl ( \stab_{\gl(n, \mathbb Z)}(\mathsf M) \bigr )^t  \] 
where $t$ denotes transposition and $\stab_{\gl(n, \mathbb Z)}(\mathsf M)$ the stabilizer of $\mathsf M$ in $\gl(n, \mathbb Z)$ with respect to the bivector representation 
\begin{equation*} \displaystyle\extp^2 : \gl(n, \mathbb Z) \rightarrow \gl(N, \mathbb Z), \ \ \ \ \  \  A \rightarrow \wedge^2 (A) 
\end{equation*}  
of $\gl(n, \mathbb Z)$ and the  action of $\gl(N,\mathbb Z)$ on the space $\mathrm{Mat}_{N, l}(\mathbb Z)$ by left multiplication. In other words
\[  \aut (\mathbb Z^n, \lambda) = \{ A^t \mid A \in \gl(n,\mathbb Z)\ \mathrm{and} \  (\wedge^2 A)\mathsf M = \mathsf M   \}. \] 
\end{theorem_A}


\begin{rem}
Although the matrix $\mathsf{M}$  depends on the choice of a $\mathbb Z$-basis in the group $\Lambda$ the group $\stab_{\gl(N, \mathbb Z)}(\mathsf M)$ (and consequently $\stab_{\gl(n, \mathbb Z)}(\mathsf M)$) is independent of such a choice. Indeed for a change of basis matrix $P \in \gl(l, \mathbb Z)$ we have $\mathsf M' = \mathsf MP$ and then $X \mathsf M' = \mathsf M'$ if and only if $X\mathsf M = \mathsf M$ for $X \in \gl(N, \mathbb Z)$.
\end{rem}
Theorem A allows us to determine the non-scalar automorphism group $\aut(\mathbb Z^n, \lambda)$ thus:

 \emph{Step 1}: We first determine the group $\mathscr N: =   \stab_{\gl(N, \mathbb Z)}(\mathsf M)$ (where $N = \binom{n}{2}$).  This is easily done if $\mathsf M$ is brought to the Smith normal form. 

\emph{Step 2}: Next we find the subgroup $\mathscr N_1 : = \mathscr N \cap \extp^2(\gl(n, \mathbb Z))$. This can be done by noting as in \cite{Nem} that the projective transformations of the projective space $\mathrm{P}(\wedge^2 {\bar{\mathbb Q}}^n)$ (here $\bar {\mathbb Q}$ stands for the algebraic closure of $\mathbb Q$) arising from the matrices in $\extp^2(\gl(n, \mathbb Z))$ preserve the projective grassmannian variety $\mathrm{Gr}(2,n)$. More algebraically, the group $\extp^2(\gl(n, \mathbb Z))$ is identified in \cite{VP07} with the connected component of the subgroup of
$\gl(N, \mathbb Z )$ preserving the ideal $\plu$ generated by all the Pl\"{u}cker polynomials 
defining the variety $\mathrm{Gr}(2,n)$. 

\emph{Step 3}:
Once $\mathscr N_1$ has thus been determined we must calculate $N = \bigl (  \extp^2 \bigr )^{-1}(\mathscr N_1)$. This amounts to calculating the exterior root of a matrix in $\mathscr N_1$ which is determined uniquely up to minus sign~\cite{LoLu} and for which a polynomial-time algorithm (implemented in GAP) exists~\cite{Grnhl}.      
 
\begin{rem} 
The algorithm mentioned in Step 3 can also determine (in polynomial time) if a matrix $A \in \gl(N, \mathbb Z)$ has an exterior root and consequently Step 2 may not be necessary. However it may still be helpful in the determination of the non-scalar automorphism group.
\end{rem}  

Until now the nonscalar automorphism group has been calculated (\cite{OP1995, NeebKH2008}) in the case $n \ge 3$ or in the case when the multiparameters are (multiplicatively) independent.
Examples of the above approach applied to the calculation of the non-scalar automorphism group of quantum tori for $n = 4$ are given in Section \ref{aut-qtor}. One such example is as follows.
\begin{exmpl}
 For $n = 4$ suppose that the group $\Lambda$ is freely generated by the commutators $q_{14}$, $q_{23}$, $q_{24}$ and $q_{34}$, while 
\[ q_{12} = q_{13} = 1.\]
Then 
\[ \aut(\mathbb Z^4, \lambda) = \Biggl \{ f_{b, \epsilon} : = \begin{pmatrix} \epsilon & 0 & 0 & -b \\
0 & \epsilon & 0 & 0 \\
0 & 0 & \epsilon & 0 \\
0 & 0 & 0 & \epsilon \end{pmatrix}, \ \  b \in \mathbb Z, \epsilon \in \{ \pm 1 \} \Biggr \}.\]
\end{exmpl}
The following proposition is also established in Section \ref{aut-qtor}.
\begin{prop_spc} Let $n \ge 5$.
For a quantum torus $\mathcal O_{\mathfrak q} (\mathbb (F^\ast)^n )$ such that \[ q_{12} = q_{13} = \cdots =  q_{1(n-2)} = 1 \]  and the remaining multiparameters  independent in $\mathbb F^\ast$ we have
\[ \aut(\mathbb Z^n, \lambda) \cong \mathbb Z_2.\] 
\end{prop_spc}
We hope that the non-scalar automorphism group can similarly be determined in many new and interesting situations using the approach described above. 
Using Theorem A we can also easily recover the group $\aut(\mathbb Z^n, \lambda)$ in the limited number of known cases~\cite{KPS94, OP1995, NeebKH2008} dealing with small values of $n$ and the case where the group $\Lambda$ has the maximal possible rank. 

We next consider the automorphism groups of quantum affine spaces. 
Until now in the study of automorphisms of quantum affine spaces the following important cases have been dealt with: (i) the uniparameter case, that is, $q_{ij} = q$ and  the group $\langle q \rangle$ is infinite cyclic~\cite{AC1992} and (ii) the case in which the multiparameters are in general position~\cite{ART1997}, that is, the group $\Lambda$ has maximum possible (torsion-free) rank. The main conclusions here are that under these assumptions the automorphism group shrinks to the image of the action of the torus ${(\mathbb F^\ast)}^n$. We show the same conclusion remains true in somewhat more general situations:
\begin{thm_B}\label{thmB}
 A quantum affine space  $\mathcal O_{\mathfrak q}(\mathbb F^n)$ such that the subgroup $\Lambda$ has rank no smaller than $\binom{n - 1}{2} + 1$ satisfies 
 \[ \aut(\mathcal O_{\mathfrak q}(\mathbb F^n)) = (\mathbb F^\ast)^n. \]  
\end{thm_B}    

An automorphism of the algebra $\mathcal O_{\mathfrak q}$ is called linear if it stabilizes the subspace spanned by $X_1, \cdots, X_n$. The following proposition gives a criterion for a multiparameter quantum affine space to be free of linear automorphisms other than those arising from the action of the torus ${(\mathbb F^\ast)}^n$.

\begin{prop_spc2}
Suppose that $\chc(\mathbb F) \ne 2$ and $n \ge 3$. Let $\mathfrak q = (q_{ij})$ be a multiplicatively antisymmetric matrix whose entries satisfy
\begin{align*}
 &\mbox{(i)}\  q_{ij} \ne q_{kl} \ \ \ \ \ \ \ \ \forall i < j , \ \ \forall k <  l, \\
 &\mbox{(ii)} \ q_{ij}q_{kl} \ne 1 \ \ \ \ \ \ \ \ \ \forall i < j , \ \ \forall k <  l, \ (i,j) \ne (k,l).
\end{align*}
 The group of linear automorphisms of the quantum affine space $\mathcal O_{\mathfrak q}(\mathbb F^n)$ coincides with the torus ${(\mathbb F^\ast)}^n$.
\end{prop_spc2}

\begin{notation}   We will use the short forms $\mathcal {O}_{\mathfrak q}$ for the quantum affine space  and $\widehat {\mathcal {O}}_{\mathfrak q}$ for quantum torus obtained from it by localizing at the multiplicative subset generated by the indeterminates $X_i$ ($1 \le i \le n$) and refer to it as the ``corresponding quantum torus".  
Automorphisms of the algebras we consider will always be $\mathbb
 F$-automorphisms. 
\end{notation}

\section{Automorphisms of quantum tori}\label{aut-qtor}

As noted above in \eqref{Neeb-Key-Result1} the essential question here is the determination of the group $\aut(\mathbb Z^n, \lambda) \le \gl(n, \mathbb Z)$ of non-scalar automorphisms. 
The following fact shown in \cite{OP1995} is the basis for our assumption that the group $\Lambda$ is torsion-free. 
\begin{lem}[Lemma 3.3(ii)]\label{Lbda_crul_cse}
Let $p$ denote the size of the torsion subgroup of $\Lambda$. The subalgebra $\widehat{\mathcal O'}$ of $\widehat{\mathcal O}_{\mathfrak q}$ generated by the powers $X_i^{\pm p}$ of the indeterminates $X_i$ is a characteristic sub-algebra of the same rank. Moreover $\widehat {\mathcal O}_{\mathfrak q}$ is free left $\widehat{\mathcal O'}$-module of finite rank and the corresponding $\lambda$-group $\Lambda'$ associated with $\widehat{\mathcal O'}$ is torsion free.
\end{lem}
 
 In view of Lemma \ref{Lbda_crul_cse}
 we will assume throughout this section that the group $\Lambda$ is torsion-free.
 We recall that for a given matrix $A \in \gl(n, \mathbb Z)$ the exterior square $\wedge^2 A$ of $A$ is the $\binom{n}{2} \times \binom{n}{2}$-matrix whose rows and columns are indexed by the pairs $(ij)\ (1 \le i < j \le n)$ ordered lexicographically and whose $((ij), (kl))$ entry is the $2 \times 2$-minor corresponding to rows $i,j$ and columns $k,l$. 
 
 By the well-known Cauchy-Binet formula the map $A \mapsto \wedge^2 A$ is multiplicative and satisfies 
 $\det (\wedge^2 A) = (\det A)^{n - 1}$~\cite{HJ}. We also have  $\wedge^2 A^t = (\wedge^2 A)^t$ whete $t$ denotes transposition.
 With $\mathsf M$ as defined in Notation \ref{not1} we have the following.

\begin{theorem_A}
For a quantum torus $\mathcal O_{\mathfrak q} (\mathbb (F^\ast)^n )$ suppose that the group $\Lambda$ generated by the multiparameters $q_{ij}$ is torsion-free. Set $N  =  \binom{n}{2}$ and let $\mathsf M$ be as defined in Notation \ref{not1} above.  Then 
\[  \aut (\mathbb Z^n, \lambda) = \bigl ( \stab_{\gl(n, \mathbb Z)}(\mathsf M) \bigr )^t  \] 
where $t$ denotes transposition and $\stab_{\gl(n, \mathbb Z)}(\mathsf M)$ the stabilizer of $\mathsf M$ in $\gl(n, \mathbb Z)$ with respect to the bivector representation 
\begin{equation*} \displaystyle\extp^2 : \gl(n, \mathbb Z) \rightarrow \gl(N, \mathbb Z), \ \ \ \ \  \  A \rightarrow \wedge^2 (A) 
\end{equation*}  
of $\gl(n, \mathbb Z)$, that is, 
\[ \stab_{\gl(n, \mathbb Z)}(\mathsf M) = \{ A \in \gl(n,Z) \mid (\wedge^2 A)\mathsf M = \mathsf M  \}. \] 
\end{theorem_A}

\begin{proof}
Writing the group $\Lambda \le \mathbb F^{\ast}$ additively, in view of Notation \ref{not1} we have   
\begin{equation}\label{expn_in_gens}
    \lambda(\mathbf e_i, \mathbf e_j) = {m_1}^{(ij)} p_1 +   \cdots + {m_l}^{(ij)} p_l, \ \ \ \ \ \ \ \ \forall 1 \le i < j \le n,
\end{equation} 
where $m_{k}^{(ij)} \in \mathbb Z$.  Thus $\mathsf M = (a_{(ij)s})$ where $a_{(ij)s} = m_s^{(ij)}$.
Now let \[ \mathscr M = (m_{ij}) \in \gl(n, \mathbb Z) \] be such that $\mathscr M^t \in \aut(\mathbb Z^n, \lambda)$.
Setting  \[ \mathbf e_j' = \mathscr M^t \mathbf e_j=  m_{j1}\mathbf e_1 + \cdots +  m_{jn}\mathbf e_n \]
we note that since $\lambda$ is an alternating function therefore $\lambda(\mathbf e_i', \mathbf e_j')$ may be expressed (e.g., \cite[Section 1.2]{MP}) in terms of $\lambda(\mathbf e_i, \mathbf e_j)$ where $i < j$ as follows:
\begin{equation}\label{transf_commtr}
    \lambda(\mathbf e_i', \mathbf e_j') = m_{ij, 12}\lambda(\mathbf e_1, \mathbf e_2) + {m_{ij, 13}}\lambda(\mathbf e_1, \mathbf e_3)  + \cdots + m_{ij, (n-1)n}\lambda(\mathbf e_{n - 1}, \mathbf e_n).  
\end{equation} 
We also note that the coefficients appearing in the RHS of the above expression constitute row (ij) of the matrix $\wedge^2(\mathscr M)$. 
Since $\mathscr M^t$ is $\lambda$-preserving by \eqref{form_prsvng} we have
 \[ \lambda(\mathbf e_i', \mathbf e_j') = \lambda(\mathbf e_i, \mathbf e_j) \ \ \ \ \ \ \ \forall  1 \le i < j \le n. \] 
Expanding and comparing the coefficients of $p_s (s = 1, \cdots, l)$ in both sides of the last equation we get noting 
\begin{equation} \label{coeff_compr}
    m_{ij, 12}m_s^{(12)} + m_{ij, 13}m_s^{(13)} + \cdots + m_{ij, (n-1)n}m_s^{((n-1)n)} =  m_s^{(ij)} \ \ \ \ \ \forall   s = 1, \cdots, l.
\end{equation} 
Letting $(ij)$ take values in the set
\begin{equation}%
 \{(12), (13),\cdots, (n - 1)n \} %
\end{equation}
in equation (\ref{coeff_compr}) and setting \[ \mathsf M^{(s)}  =  (m_s^{(12)}, m_s^{(13)}, \cdots, m_s^{((n - 1)n)})^t \] 
we thus get 
\[ \wedge^2(\mathscr M)\mathsf M^{(s)} = \mathsf M^{(s)} \ \ \ \ \  \ \forall s \in 1, \cdots, l. \]
As $\mathsf M^{(s)}$ is column $s$ of $\mathsf M$ (as defined in Notation \ref{not1}) it follows that:
\[ \wedge^2(\mathscr M)\mathsf M = \mathsf M. \]
Clearly the above reasoning is
 reversible. This establishes the assertion of the theorem.
\end{proof}
\begin{rem}\label{key-rmk}
In the situation of Theorem A since the $N$ elements $\lambda(\mathbf e_i, \mathbf e_j)$ ($1 \le i < j \le n$) generate the group $\Lambda$ therefore the columns of $\mathsf M$ must span a free $\mathbb Z$-submodule of $\mathbb Z^{N}$ of rank $l$. Thus $\mathsf M$ must fix a free submodule of $\mathbb Z^{N}$ of rank $l$.
\end{rem}
The remainder of this section is devoted to the determination of the 
automorphism group of the quantum torus defined by multiparameters satisfying a given set of  conditions. Our line of approach as suggested by Theorem A for determining the matrices $A \in \gl(n, \mathbb Z)$ whose transpose lies in the group $\aut(\mathbb Z^n, \lambda)$ can be summarized as follows.
\begin{equation}
    \boxed{\mbox{Relations matrix $\mathsf M$}}  \longrightarrow \boxed{\mbox{Submodule Fixed by $\wedge^2 A$}} \longrightarrow \boxed{\mbox{$\wedge^2 A$}} \longrightarrow  \boxed{A}.
\end{equation} 
\begin{example}
For $n = 2$, the group $\Lambda$ is infinite cyclic with the generator $q = q_{i2}$ and $\mathsf M$ is the $1 \times 1$ matrix $1$. Clearly, the stabilizer of $\mathsf M$ is $\{1\}$ and thus $\aut(\mathbb Z^2, \lambda) = \mathrm{SL}(2, \mathbb Z)$.   
\end{example}

\begin{example}[\cite{OP1995}] \label{n3r2} 
For $n = 3$ the group $\aut(\mathbb Z^3, \lambda)$ was discussed in \cite{OP1995}. The most non-trivial situation is obtained when $\Lambda \cong \mathbb Z^2$. As shown in the proof of \cite[Proposition 3.7]{OP1995} using a change of variables we may suppose in this case that $q_{12} = 1$ and the set $\{ q_{13}, q_{23} \}$ is independent where $q_{ij} = [X_i, X_j]$ and moreover  
\[  \aut(\mathbb Z^3, \lambda) =   \ \Biggl \{ \begin{pmatrix}
\epsilon & 0 & a \\
0 & \epsilon & b \\
0 & 0 & \epsilon
\end{pmatrix} 
,  \ a,b \in \mathbb Z, \ \epsilon \in \{\pm 1\} \Biggr \}. \] 
Note that in this case   
\[ \mathsf M : = \begin{pmatrix}
0 & 0 \\
 1 & 0 \\
 0 & 1
\end{pmatrix} \ \ \mathrm{and \ therefore \  for}\ A \in \aut(\mathbb Z^3, \lambda), \ \ \wedge^2 A^t = \begin{pmatrix} 1 & 0 & 0 \\
\ast & 1 & 0 \\
\ast & 0 & 1
\end{pmatrix}.   \]
We also note that 
\[ \wedge^2  \begin{pmatrix}
\epsilon & 0 & a \\
0 & \epsilon & b \\
0 & 0 & \epsilon
\end{pmatrix}^t =  \begin{pmatrix}
1 & 0 & 0 \\
b\epsilon & 1 & 0 \\
-a\epsilon & 0 & 1
\end{pmatrix}  \]
as would be expected from Theorem A.
\end{example}
As just seen in Example \ref{n3r2} for $n = 3$ and $\Lambda \cong \mathbb Z^2$ a change of variables leads to a simple set of relations for the commutators $[X_i, X_j]$. In turn this means a simple form for the relations matrix $\mathsf M$ in Theorem A thus facilitating the computation of the non-scalar automorphism group in this case.

The situation is more complex for $n = 4$.
For example, if $n = 4$ and $\Lambda \cong \mathbb Z^5$ one msy expect that a suitable change of variables will allow us to assume, for example, that $q_{12} = 1$. However this is not true by the example of ~\cite[Section 3.11]{MP} and is still not true when $\Lambda \cong \mathbb Z^4$~\cite{GQ06}. 

\begin{prop}\label{smp-cse-4-2}
For $n = 4$ suppose that the group $\Lambda$ is freely generated by the commutators $[X_1, X_4]$, $[X_2, X_3]$,$[X_2, X_4]$ and $[X_3, X_4]$, while 
\[ [X_1, X_2] = [X_1, X_3] = 1.\]
Then 
\[ \aut(\mathbb Z^4, \lambda) = \Biggl \{ f_{b, \epsilon} : = \begin{pmatrix} \epsilon & 0 & 0 & -b \\
0 & \epsilon & 0 & 0 \\
0 & 0 & \epsilon & 0 \\
0 & 0 & 0 & \epsilon \end{pmatrix}, \ \  b \in \mathbb Z, \epsilon \in \{ \pm 1 \} \Biggr \}.\]
\end{prop} 

\begin{proof}
In this case for $A^t \in \aut(\mathbb Z^4, \lambda)$ 
using Theorem A we see that the columns 
$C^{(i)}, \ i = 1, \cdots, 6$ of $\wedge^2 A$ satisfy $C^{(i)} = I^{(i)}, \ \forall i = 3, \cdots, 6$ where $I^{(i)}$ stands for the $i$-th column of the identity matrix $I_6$. Thus only the first two columns need to be determined. To this end we note that by \cite{Nem} the image of $\wedge^2 A$ under the projection $\rho: \gl(6, \bar {\mathbb Q}) \rightarrow \pgl(6, \bar {\mathbb Q})$ preserves 
the projective grassmannian variety $\mathrm{Gr}(2,4)$ embedded in the projective space $\mathrm{P}(\wedge^2 \bar{\mathbb Q}^4)$ where 
$\bar {\mathbb Q}$ stands for the algebraic closure of $\mathbb Q$. 

By a well-known fact~\cite[Section 14.7]{Berg} we know that the group of the projective quadric $\mathrm{Gr}(2,4)$ is the image of the isometry group $\mathrm{O}(\beta)$ in the projective group $\pgl(6, \bar {\mathbb Q})$ where $\beta$ is the 
polarization of the Pl\"{u}cker quadratic form \[q(\xi_{12}, \xi_{13}, \xi_{14}, \xi_{23}, \xi_{24}, \xi_{34}) =  \xi_{12}\xi_{34} 
- \xi_{13}\xi_{24} + \xi_{14}\xi_{23}. \]

It is easily checked that the matrix of the function $\beta$ with respect to the basis $\mathbf e_i \wedge \mathbf e_j \  (i < j)$ is
the matrix
\[ P:= \begin{pmatrix} 0 & 0 & 0 
& 0 & 0 & 1 \\ 0 & 0 & 0 & 0 & -1 & 0 \\ 0 & 0 & 0 & 1 & 0 & 0 \\ 0 & 0 & 1 & 0 & 0 & 0 \\ 0 & -1 & 0 & 0 & 0 & 0 \\ 1 & 0 & 0 & 0 & 0 & 0 \end{pmatrix}. \]
Thus  $(\wedge^2 A) C \in \mathrm{O}(\beta)$ for some scalar matrix $C \in \gl(6,\bar{\mathbb Q})$.
Therefore $\wedge^2 A$ must satisfy  
\begin{equation}    
\label{iso-Pl-fm}
C^2(\wedge^2 A)^t P \wedge^2 A = P.
\end{equation}
As \[ \det \wedge^2 A = (\det A)^3 = \pm 1 \] it is clear from the preceding equation that $\det C = \pm 1$. Writing $C = \diag(\lambda, \cdots , \lambda)$ this means that either $\lambda$ is root of the polynomial $\Phi_-: = Y^6 - 1$ or $\Phi_+: = Y^6 + 1 = 0$.    
Since the matrices $\wedge^2 A$ and $P$ have integer entries it is clear from \eqref{iso-Pl-fm} that $\lambda = \pm 1$ and thus $C = \pm I$. This means that $\wedge^2A  \in \mathrm{O}(\beta)$. Direct calculation using \eqref{iso-Pl-fm} reveals that $\wedge^2 A$ has the form
\[ \wedge^2 A =  \begin{pmatrix} 1 & 0 &0 &0 &0 &0 \\ 0 & 1 & 0 & 0 & 0 & 0 \\ 0 & 0 & 1 & 0 & 0 & 0 \\  0 & 0 & 0 & 1 & 0 & 0 \\ b & 0 & 0 & 0 & 1& 0 \\ 0 & b & 0& 0& 0& 1 \end{pmatrix}, \ \ \ \ \ \ \ \ \   b \in \mathbb Z.\]
As is readily checked the last equation means noting \cite[Corollary 2]{LoLu} that 
\[ A = \epsilon \begin{pmatrix} 1 & 0 & 0 & 0 \\
0 & 1 & 0 & 0 \\
0 & 0 & 1 & 0 \\
-b & 0 & 0 & 1 
\end{pmatrix}, \ \ \ \ \ \epsilon = \pm 1.
\]
and thus \[ \aut(\mathbb Z^4, \lambda) = \Biggl \{ f_{b, \epsilon} : = \begin{pmatrix} \epsilon & 0 & 0 & b \\
0 & \epsilon & 0 & 0 \\
0 & 0 & \epsilon & 0 \\
0 & 0 & 0 & \epsilon \end{pmatrix},\ \ \ \  b \in \mathbb Z, \epsilon \in \{ \pm 1 \} \Biggr \}.\]
 \end{proof}

\begin{example}
Following a similar approach as in Proposition \ref{smp-cse-4-2} we can show that for $n = 4$ assuming that the group $\Lambda$ is freely generated by the commutators  $[X_2, X_3]$,$[X_2, X_4]$ and $[X_3, X_4]$, while 
\[ [X_1, X_2] = [X_1, X_3] =  [X_1, X_4] = 1\]
the non-scalar automorphism group is given in this case by \[ \aut(\mathbb Z^n, \lambda) =  \Biggl \{ \phi_{a,b, \epsilon} : = \begin{pmatrix} \epsilon & a & b & 0 \\
0 & \epsilon & 0 & 0 \\
0 & 0 & \epsilon & 0 \\
0 & 0 & 0 & \epsilon \end{pmatrix}, \ \  a, b \in \mathbb Z, \epsilon \in \{ \pm 1 \} \Biggr \}. \] 
\end{example}

\begin{rem}\label{detmn-stab}
In a more general situation where $\mathsf M$ does not have a simple form as seen in the above examples it is easily checked that \[ \stab_{\gl(N, \mathbb Z)}(\mathsf M) = U^{-1}(\stab_{\gl(N, \mathbb Z) }( U\mathsf M V)U \] where $U\mathsf M V$ is the Smith normal form $\snf(\mathsf M)$ of $\mathsf M$.  

\end{rem}
\begin{prop}
For $n = 4$  suppose that the multiparamters $q_{ij}\ (i,j) \ne (1,2)$ are  independent for $i < j$  and \[ q_{12} = \prod_{i < j, \ (i,j) \ne (1,2)} q_{ij}. \]
Then $\aut(\mathbb Z^4, \lambda) = 
\{\pm I_4 \}$.
\end{prop}
\begin{proof}
Clearly, in this case the relations matrix $\mathsf M$ has the form 
\[ \mathsf M = \begin{pmatrix} 1 & 1 & 1 & 1 & 1 \\ 1 & 0 & 0 & 0 & 0 \\ 0 & 1 & 0 & 0 & 0 \\ 0 & 0 &1 & 0 & 0 \\ 0 & 0 & 0 & 1 & 0 \\ 0 & 0 & 0 & 0 & 1 \end{pmatrix}  \] 
Using the smith form calculator~\cite{keithmatt}  we find
\[ 
 \snf(\mathsf M) =  U \mathsf M V = \begin{pmatrix}
1 & 0 &	0 &	0 &	0 \\
0 &	1 &	0 &	0 &	0 \\
0 &	0 &	1 &	0 &	0 \\
0 &	0 & 0 &	1 & 0 \\
0 &	0 & 0 &0 & 1 \\
0 &	0& 0 & 0& 0
\end{pmatrix}, \ \ \ \ \ \   U =  \begin{pmatrix} 0 & 1 &  0 & 0 
& 0 & 0 & \\ 0 & 0 & 1 & 0 & 0 & 0 \\ 0 & 0 & 0 & 1 & 0 & 0 \\ 0 & 0 & 0& 0& 1 & 0 \\ 0 & 0 & 0 & 0 & 0 & 1 \\ -1 & 1& 1 & 1& 1& 1  \end{pmatrix},  \ \ \  V  = I_5. \] 

Clearly,
\[ \stab_{\gl(N,\mathbb Z)}(U\mathsf M V) = \begin{pmatrix} 1 & 0 & 0 & 0 & 0 & x \\
0 &  1 & 0 & 0 &  0 & y \\
0 & 0 &  1 & 0 & 0 & z \\
0 & 0 & 0 &  1 &  0 & u \\
0 & 0 & 0 & 0 & 1 & v \\
0 & 0 & 0 & 0 & 0 & 1 
\end{pmatrix},  \ \ \ \ \ \  x,y,z, u, v \in \mathbb Z .\] 
Using Remark \ref{detmn-stab} and calculating with the help of SageMath~\cite{sagemath} we find that  \[ \stab_{\gl(N,\mathbb Z)}(\mathsf M) =
\begin{pmatrix}
1 -S   &   S   &   S   &   S   &  S   &   S\\
-x  &   x + 1   &  x & x &  x  & x \\
-y   &  y    &  y + 1  & y  &   y   &   y \\
-z & z  &  z &  z + 1   &  z &  z \\
-u  &   u  & u  &  u &  u + 1  & u\\
-v  &   v  & v  &  v  & v  & v + 1 
\end{pmatrix}, \ \ \ S =  u + v + x + y + z.
\]


For a matrix $B \in \extp^2 (\gl(n, \mathbb Z) \cap \stab_{\gl(N, \mathbb Z)}(\mathsf M)$ by the same reasoning as in Proposition \ref{smp-cse-4-2} we obtain 
$B^t P B = P$. Comparing the first rows in both sides we obtain 
\begin{align*}
    T -2v &= 0 \\
    u + v - T &= 0 \\
    v - z - T &= 0 \\
    v - y - T &= 0 \\
    v + x - T &= 0 \\
    u + x + y + z + T &= 0  
\end{align*}
where $T =  2uv + 2v^2 - 2ux + 2vx + 2vy + 2vz + 2yz$. It is easily seen that this system has a unique solution $x = y = z = u = v = 0$. This completes our proof. 

\end{proof}
 
It was shown in \cite{OP1995} that if the $\binom{n}{2}$ multiparameters $q_{ij} (1 \le i < j \le n)$ are independent in $\mathbb F^\ast$ then $\aut(\mathbb Z^n, \lambda) \cong \mathbb Z_2$. With the help of Theorem A we show in the next proposition that the same conclusion remains valid under a somewhat weaker hypothesis. 

\begin{prop}\label{
Theorem-B}
For a quantum torus $\mathcal O_{\mathfrak q} (\mathbb (F^\ast)^n )$ with the $n - 3$ multiparameters \[ q_{12},q_{13}, \cdots, q_{1(n-2)}\] set to one and the remaining multiparameters \[ q_{1(n-1)}, q_{1n}, q_{23}, q_{24}, \cdots, q_{(n -1)n} \] independent in $\mathbb F^\ast$
\[ \aut(\mathbb Z^n, \lambda) \cong \mathbb Z_2\] 
and consequently by \eqref{Neeb-Key-Result1} we have  \[ 1  \rightarrow \Hom(\mathbb Z^n, {\mathbb {F}}^\ast) \rightarrow \aut(\mathcal O_{\mathfrak q}({F^\ast}^n)) \rightarrow \mathbb Z_2 \rightarrow 1. \] 

\end{prop}

\begin{proof}
From the theorem hypothesis it is evident that in this case the relations matrix $\mathsf M$ of Theorem A   
is obtained from the identity matrix $I_{\binom{n}{2}}$ by deleting  the first $n - 3$ columns.
Consequently, any matrix $B \in \stab(\mathsf M)$ must coincide with $I_{\binom{n}{2}}$ in all but the first $n - 3$ columns, that is, $B$ must have the form
\begin{equation}
\label{form-of-B}
B = 
\left(
\begin{BMAT}(c){ccc}{ccc}
\ast \cdots \ast & & \\ 
 \ast \cdots \ast & \begin{BMAT}(c){ccc}{ccc}
 & & \\
 & 1 & \\
 & & 1 
\end{BMAT} &  \\
\ast \cdots \ast &  &  I_{\binom{n - 1}{2}} 
\end{BMAT}
\right ).
\end{equation}

Suppose that $B = (b_{ij, kl})$ is induced from a matrix $A := (a_{uv}) \in \gl(n, \mathbb Z)$, that is, $B = \wedge^2 (A)$. 
The $(n - 1) \times (n - 1)$-submatrix $A^+$ of $A$ formed by the rows $2, \cdots, n$ and columns  $2, \cdots, n$ satisfies
$\wedge^2 (A^+) = I_{\binom{n-1}{2}}$. 
As is well known \[ \det(\wedge^2 (A^+)) =  \det(A^+)^{n -2} \] 
and hence $A^+$ is also nonsingular. Applying  
\cite[Corollary 2]{LoLu} we obtain 
\begin{equation}\label{sbmatxfrm}
A^+ = \epsilon  I_{n - 1}, \ \ \ \ \ \ \ \  \epsilon  \in \{-1, +1\}.  
\end{equation}     

In view of \eqref{form-of-B} we clearly have
\begin{equation*} b_{12,2j} = 0 \ \ \ \ \ \ \ \ \ \ \forall j = 3, \cdots, n .\end{equation*} 
Computing the  $2 \times 2$ minors of $A$  corresponding to the $(12, 2j)$ entries of $B$ with the help of \eqref{sbmatxfrm} it is immediately seen that the $1j$ entry $a_{1j}$ of $A$ equals to zero. Similarly, $b_{13,23} = 0$ implies $a_{12} = 0$. 
By the same token using the evident fact that
\[ b_{jn,1n} = 0 \ \ \ \ \ \ \ \ \ \  \forall  j = 2, \cdots, n - 1 \] we obtain $a_{j1} = 0$. Again $b_{(n-1)n, 1(n-1)} = 0$ implies $a_{n1} = 0$.  

It now follows that $A$ is diagonal and moreover $A \in \gl(n, \mathbb Z)$ implies 
\begin{equation}\label{} 
a_{11} =  \epsilon',  \ \ \ \ \ \ \ \ \ \ \epsilon' \in \{-1, +1\}
\end{equation}
But $\epsilon = \epsilon'$ as otherwise \[ b_{(1n),(1n)} = \epsilon'\epsilon = - 1 \]  
which contradicts the form of $B$ as noted in \eqref{form-of-B}.
Thus $A = \epsilon I_n$ and the proof is complete.   
\end{proof}

\section{Automorphisms of quantum affine spaces - preliminary facts}
Automorphisms of the quantum affine spaces  $\mathcal O_{\mathfrak q}$ were considered in \cite{AC1992, ART1997, OP1995}. In our theorems to follow we shall be utilizing the definitions, facts and results in these articles which we briefly recall here and also add a few easy propositions.
\begin{defn}[Section 1.4 of \cite{AC1992}]
An automorphism $\sigma$ of $\mathcal O_{\mathfrak q}$ is called linear if it has the form
\begin{equation}\label{lin_aut_def} 
\sigma(X_i)= \sum_{j = 1}^n \alpha_{ij}X_j \ \ \ \ \ \ \ \forall i \in \{1, \cdots, n\}, \ \ \ \ (\alpha_{ij}) \in \gl(n, \mathbb F).
\end{equation}
\end{defn}
In other words, an automorphism is linear if and only if it preserves the degree one component of 
the algebra $\mathcal O_{\mathfrak q}$ with respect to the $\mathbb N_0$-grading by total degree.
The subgroup of linear automorphisms is denoted as $\aut_{\mathrm L}(\mathcal O_{\mathfrak q})$.
For a matrix $(\alpha_{ij}) \in \gl(n, \mathbb F)$ to define an automorphism as in \eqref{lin_aut_def} the   
following necessary and sufficient conditions must hold (\cite{AC1992}):
\begin{equation}\label{AC-cond-lin-aut}
\alpha_{ik} \alpha_{jl}(1 - q_{ij} q_{lk}) = \alpha_{il} \alpha_{jk}(q_{ij} - q_{lk})  \ \ \ \ \ \ \ \forall i < j , \ \ \forall k \le l.  
  \end{equation}
The last equation may be re-written as 
\begin{equation}\label{AC-cond-lin-aut-re}
   \alpha_{ik} \alpha_{jl}(q_{kl} - q_{ij}) = \alpha_{il} \alpha_{jk}(q_{kl}q_{ij} - 1)  \  \ \ \ \ \ \forall i < j , \ \ \forall k \le l. 
\end{equation}
Setting $k = l$ in the last equation we obtain 
\begin{equation}\label{setkeql}
\alpha_{ik} \alpha_{jk} (1 - q_{ij}) = \alpha_{ik} \alpha_{jk}(q_{ij} - 1)    \ \ \ \ \  \ \forall i < j, \ \ \forall k \in \{1, \cdots, n\} . 
\end{equation}

\begin{observation}\label{avd-1}
Clearly, the last equation means that if $\chc(\mathbb F) \ne 2$ and none of the multiparameters $q_{ij} (i < j)$ equals to unity  
then at least one of the coefficients $\alpha_{ik}$ and $\alpha_{jk}$ vanishes.
It is immediate that in this case the nonsingular matrix $(\alpha_{ij})$ has exactly one nonzero entry in each row and each column. 
\end{observation}

The next proposition is an easy consequence of the preceding observation.

\begin{prop}\label{simple-cse-q(ij)-not-1}
Suppose that $\chc(\mathbb F) \ne 2$ and the entries of $\mathfrak q$ satisfy \begin{equation} \label{sim-cond1}
q_{ij} \ne 1 \ \ \ \ \ \ \ \ \ \ \  \  \forall 1 \le i < j \le n.  \tag{$\blacklozenge$}
\end{equation}
Then 
\begin{equation}\label{lin_aut}
\aut_{\mathrm L}(\mathcal O_{\mathfrak q}) \cong {\mathbb F^\ast}^n \ltimes \mathcal{P}  
\end{equation}
for a subgroup $\mathcal P$ of $S_n$.
\end{prop}

Besides the rather simple condition in \eqref{sim-cond1} there are other situations where \eqref{lin_aut} holds. Indeed as shown in \cite{OP1995} this is the case if the following condition is satisfied:

\begin{equation} \label{OP-cond} \tag{$\lozenge$}
\mathrm{The\ \  localization}\ \  \widehat{\mathcal  O}_{\mathfrak q} \  \ \mathrm{of} \ \  \mathcal  O_{\mathfrak q} \ \  \mathrm{has\ \ center} \ \  \mathbb F.
\end{equation}
In view of the foregoing we make the following definition.

\begin{defn}
Given a quantum affine space $\mathcal  O_{\mathfrak q}$ a permutation $\pi \in S_n$ is said to be \emph{admissible} if  $\pi \in \aut(\mathcal  O_{\mathfrak q})$.
\end{defn}

By the remark following \cite[Proposition 3.2]{OP1995} a permutation $\pi$ is admissible if and only if the any of the following equivalent pair of conditions holds for the corresponding permutation matrix $\mathfrak p$.

\begin{equation}\label{OP-criterion2}
\mathfrak p \mathfrak q \mathfrak p^t = \mathfrak q     \ \ \ \ \ \ \Leftrightarrow \ \ \ \ \ \   \mathfrak p \mathfrak q = \mathfrak q \mathfrak p. 
\end{equation}

The next lemma and proposition record some consequences of  a permutation $\pi$ being admissible. 
\begin{lem}\label{fund_lemma}
For a given quantum affine space $\mathcal O_{\mathfrak q}$ if a permutation $\pi$ is admissible then 
\begin{itemize}
   \item[(i)] \noindent  for each $r$-cycle $(j_1j_2 \cdots j_r)$, where $r \ge 2$ in the decomposition of $\pi$  
\begin{equation*}\label{q-relatn-cycle}
    q_{j_rj_1} =  q_{j_1j_2} = q_{j_2j_3} = \cdots = q_{j_{r-2}j_{r - 1}}  = q_{j_{r-1}j_{r}},   \ \ \ \ \mbox{and} 
\end{equation*}    

   \item[(ii)] \noindent for each fixed point $k$ of $\pi$ and each $r$-cycle $(j_1j_2 \cdots j_r)$ in the the decomposition of $\pi$ 
\begin{equation*}\label{f-pts}
q_{j_1k} = q_{j_2k} = \cdots = q_{j_rk}.
\end{equation*}    

\end{itemize}
\end{lem}

\begin{proof}
(i)  This follows easily by applying $\pi$ to \[ X_{j_t}X_{j_{t \dot{+} 1}} = q_{j_tj_{t \dot{+} 1}}X_{j_{t \dot{+} 1}}X_{j_t} \] where $\dot{+}$ denotes addition modulo $r$. As a result we obtain 
 \[ X_{j_{t \dot{+} 1}}X_{j_{t \dot{+} 2}} = q_{j_tj_{t \dot{+} 1}}X_{j_{t \dot{+} 2}}X_{j_{t \dot{+} 1}} \] 
for all $t \ge 1$ running modulo $r$. Part (ii) is similar to (i).
\end{proof}

\begin{prop}\label{prdt_pr_trvl}
Suppose that the group $\aut(\mathcal O_{\mathfrak q})$ contains a non-trivial permutation $\pi$.  Then there exists a relation
\begin{equation}
 q_{ij}q_{kl} = 1, \ \ \ \ \ (i,j),  (k,l) \in \{1,\cdots, n\}^2
\end{equation}
such that  $i < j$ and $k < l$. \end{prop}
\begin{proof}
If $(j_1j_2 \cdots j_r)$ is a cycle in the decomposition of $\pi$ then by Lemma \ref{fund_lemma} we have
\begin{equation}\label{cycl-relns}
q_{j_rj_1} =  q_{j_1j_2} = q_{j_2j_3} = \cdots = q_{j_{r-2}j_{r - 1}}  = q_{j_{r-1}j_{r}},     
\end{equation}
Evidently the sequence of differences $j_{t \dot{+} 1} - j_{t}$ in \eqref{cycl-relns} where $t$ varies modulo $r$ in the set $\{1, \cdots, r \}$ has both positive as well as negative terms.  We can thus find $u, v$ ($u \ne v$) and such that 
$j_{u \dot{+} 1} - j_u$ and $j_{v \dot{+} 1} - j_v$ have opposite signs.
Without loss of generality we may assume that $j_{u +1} - j_u < 0$. 
In view of \eqref{cycl-relns} we now obtain:
\[ q_{j_{v}j_{v + 1}}q_{j_{u + 1}j_{u}} = 1. \]

\end{proof}

Combining the above facts immediately yields a criterion for ensuring that all linear automorphisms of a quantum affine space result from the action of the torus. 

\begin{cor}
Suppose that $\chc(\mathbb F) \ne 2$ and the multiparameters satisfy:   
\begin{equation} \label{prdts-ne-1-stng}
 q_{ij}q_{kl} \ne 1\ \mbox{for all} \  i < j \ \mbox{and} \ k < l.    \tag{$\clubsuit$}
\end{equation} Then
\[ \aut_{\mathrm L}(\mathcal O_{\mathfrak q}) = (\mathbb F^\ast)^n. \]
\end{cor}
\begin{proof}
The condition \eqref{prdts-ne-1-stng} means that $q_{ij} \ne \pm 1$ whenever $i < j$. Thus Proposition \ref{simple-cse-q(ij)-not-1} applies and we conclude by Proposition \ref{prdt_pr_trvl}. 
\end{proof}

In the next section we will show  that the same conclusion remains valid under a different hypothesis which allows $1$ to be included as an (off-diagonal) entry of $\mathfrak q$.

\section{Automorphisms of quantum affine spaces--Main results}

\begin{prop}
Suppose that $\chc(\mathbb F) \ne 2$ and $n \ge 3$. Let $\mathfrak q = (q_{ij})$ be a multiplicatively antisymmetric matrix whose entries satisfy
\begin{align*}
 &\mbox{(i)}\  q_{ij} \ne q_{kl} \ \ \ \ \ \ \ \ \forall i < j , \ \ \forall k <  l, \\
 &\mbox{(ii)} \ q_{ij}q_{kl} \ne 1 \ \ \ \ \ \ \ \ \ \forall i < j , \ \ \forall k <  l, \ (i,j) \ne (k,l).
\end{align*}
 The group of linear automorphisms of the quantum affine space $\mathcal O_{\mathfrak q}(\mathbb F^n)$ coincides with the torus ${(\mathbb F^\ast)}^n$.
\end{prop}

\begin{proof}
Note that the hypothesis means that there is at most one entry $q_{i'j'}$ with $i' < j'$ that is equal to $1$. Similarly, there is at most one entry $q_{i''j''}$ with $i'' < j''$ that is equal to $-1$.
Suppose that \[ A = (\alpha_{ij}) \in \gl(n, \mathbb F)\] 
induces a linear automorphism of the given quantum affine space. It suffices to show that $A$ is a diagonal matrix. 
Using \eqref{setkeql} it is easily seen that in any column of $A$ at most two entries can be nonzero and in the case there are two nonzero entries, these must be in the $i'$-th and 
$j'$-th rows. 

We claim that there can be at most two columns in $A$ that have two nonzero entries. Indeed suppose that $2 + s$ columns have two nonzero entries necessarily in the $i'$-th and $j'$-th rows. As just noted the remaining $n - 2 - s$ columns each has exactly one non-zero entry.  Clearly, if $s > 0$ these $n - 2 - s$ nonzero entries in $A$ cannot fulfill the requirement of a non-zero entry in each of the $n - 2$ rows other than the $i'$-th and $j'$-th rows. This shows that $s = 0$. 

Next we note that at least one of the entries $\alpha_{i'i'}$ and  $\alpha_{j'i'}$ is non-zero. To see this we pick $m < p$ in the range $1, \cdots,n$. 
By \eqref{AC-cond-lin-aut-re} we have noting that $q_{i'j'} = 1$  
\begin{align}
\label{apln_fundl_rln1}
\alpha_{i'm}\alpha_{j'p}(q_{mp} - 1) &= \alpha_{i'p}\alpha_{j'm}(q_{mp} - 1). 
\end{align} 
 If $(m,p) \ne (i',j')$ then as noted above $q_{mp} \ne 1$ and \eqref{apln_fundl_rln1} reveals that  
the minor of $A$ corresponding to the $2 \times 2$ submatrix formed by the $i'$-th and $j'$-th rows and the $m$-th and $p$-th columns is equal to zero. If $\alpha_{i'i'} = \alpha_{j'i'} = 0$ then the minor corresponding to the $2 \times 2$ submatrix $K$ defined by the $i'$-th and $j'$-th rows and the $i'$-th and $j'$-th columns is also equal to zero.
This would mean that a row of the exterior square $\wedge^2 A$ of $A$ is the zero row contradicting the assumption $A$ is non-singular. 
By the same token at least one of the entries $\alpha_{i'j'}$ and $\alpha_{j'j'}$ in column $j$ is nonzero. 
Moreover, the non-zero entries in the two columns, namely, $i'$ and $j'$ cannot be in only one of the rows $i'$ or $j'$ as in this case the determinant of $K$ will be zero.

We now claim that if a column of $A$ has two non-zero entries then it must be the $i'$-th or the $j'$-th column. Indeed let $h$ be a column of $A$ having two non-zero entries. As noted above these non-zero entries of $h$ must be rows $i'$ and $j'$. Thus the three columns $i'$, $j'$ and $h$ have non-zero entries only in rows $i'$ and $j'$. Consequently there can be at most $n -3$ nonzero entries in columns other than $i',j'$ and $h$ that are contained in the $n - 2$ rows other than $i'$ and $j'$. But this means there is a row with no non-zero entry contradicting the assumption that $A$ is non-singular.   


We now consider a column of $A$, say the $k$-th,
that has only one non-zero entry, say in the $u$-th row.  
This is possible in view of the foregoing noting that $n \ge 3$ buy the theorem hypothesis.   

Let $\alpha_{uk}$ be the unique nonzero entry in the $k$-th column.  In view of the foregoing observations it is easily seen that $\alpha_{uk}$ is also the unique nonzero entry in the $u$-th row. 
We note as $n \ge 3$ (by the hypothesis) it is always possible to choose an $l$ ($1 \le l \ne k \le n$) such that either 
\begin{align}
& (i)\  k < l  \ \mathrm{and} \ (k, l) \ne (i'',j''); \  \mathrm{or},  \nonumber\\ 
& (ii) \ l < k  \ \mathrm{and} \ (l, k) \ne (i'',j''). \numberthis
\label{col:k-col:l}
\end{align}
note that we do not assume that the column $l$ has only one non-zero entry. \\
\textbf{Step 1:} In case (i) we suppose that the $l$-th column has a nonzero entry, namely, $\alpha_{vl}$ in the $v$-th row. We claim that $v > u$. Indeed if $v < u$ then \eqref{AC-cond-lin-aut-re} yields  
\begin{equation}\label{forbidden-comb}
    \alpha_{vk}\alpha_{ul}(q_{kl} - q_{vu}) = \alpha_{uk}\alpha_{vl}(q_{kl}q_{vu} - 1).
\end{equation}
By the preceding paragraph the LHS of \eqref{forbidden-comb} vanishes (as $\alpha_{uk}$ is the only non-zero entry in column $k$) and (noting theorem hypothesis) the RHS is nonzero, unless,  
\begin{equation}\label{avoidance_comb} 
(k,l) = (v, u) = (i'',j''). 
\end{equation}
But by the choice of the pair $(k,l)$ \eqref{avoidance_comb} cannot hold. Our assertion, namely, $u < v$ now follows.   
We now apply \eqref{AC-cond-lin-aut-re} to $u < v$ and $k < l$ and thus obtain
\[ \alpha_{uk}\alpha_{vl}(q_{kl} - q_{uv}) = \alpha_{vk}\alpha_{ul}(q_{kl}q_{uv} - 1). \] 
Clearly in the last equation the RHS is equal to zero, but the LHS can vanish only if $q_{kl} = q_{uv}$ and by the theorem hypothesis this is possible only if $k = u$ and $l = v$. The $k$-th and $l$-th columns of $A$ must therefore coincide with the corresponding columns of a suitable diagonal matrix $D$.

\textbf{Step 2:} We now suppose that case (ii) holds in \eqref{col:k-col:l}.  As before let $\alpha_{vl} \ne 0$ for some 
$v$. We claim that in this case $v < u$. For assuming the contrary and applying \eqref{AC-cond-lin-aut-re} leads to a contradiction quite similarly to that resulting from \eqref{forbidden-comb} above. Arguing as in case (i) and applying \eqref{AC-cond-lin-aut-re} to $v < u$ and $l < k$ we obtain
\[ \alpha_{vl}\alpha_{uk}(q_{lk} - q_{vu}) = \alpha_{ul}\alpha_{vk}(q_{lk}q_{vu} - 1) \]
and the last equation means that $l = v$ and $k = u$.

\textbf{Conclusion of proof:} Suppose that $k \ne i'',j''$. Then Steps 1 and 2 above immediately yield the fact that the matrix $A$ is a diagonal. Otherwise if $k = i''$ (resp. $k = j''$) by the above reasoning we obtain that the $s$-th column of $A$ has the form $\alpha_s \mathbf e_s$ (where $\mathbf e_s$ stands for the $s$-th column of the identity matrix and $\alpha_s \in \mathbb F^\ast$) except possibly when  $s = j''$ (resp. $s = i''$). But we may now redefine $k$ to be any natural number in $\{1, \cdots, n\} \setminus \{i'',j''\}$ and apply one of the Steps 1 and 2 depending on whether $k < j''$ or $k > j''$ (resp. $k < i''$ or $k > i''$).   
 
\end{proof}

As before, let $\{\mathbf e_1, \cdots, \mathbf e_n\}$ denote the standard basis in $\Gamma : = \mathbb Z^n$. Then $\{\mathbf e_i \wedge \mathbf e_j \mid 1 \le   i < j    \le n \}$ is a basis in $\wedge^2 \Gamma$. As usual, for a permutation $\pi \in S_n$ let $P \in \gl(n, \mathbb Z)$ denote the corresponding permutation matrix. 
Clearly $S_n$ acts on the $\mathbb Z$-module $\wedge^2 \Gamma$ via 
\begin{equation}\label{actn-Sn}
\pi(e_i \wedge e_j) = \wedge^2 P (e_i \wedge e_j) = e_{\pi(i)} \wedge e_{\pi(j)} \ \  \ \forall \pi \in S_n.    
\end{equation}  
\emph{In this action the image of $\pi$ in $\End_{\mathbb Z}(\wedge^2 \Gamma)$ will be denoted as $\wedge^2 \pi$.} By restriction we get an action of $S_n$ on the subset \[ \bar B =  
\{ \epsilon e_i \wedge e_j \mid  i < j, \ \mathrm{and} \ \epsilon \in \{-1, 1\} \}. \] 
We note the following lemma that will be used in the proof of Theorem B.

\begin{lem} \label{key-lem}
Let $\pi \in S_n \ (n \ge 3)$ and 
$\mathrm{Fix}(\wedge^2 \pi)$ stand for the sub-module
of $\wedge^2 \Gamma$ left fixed by the permutation $\pi$ under the action \eqref{actn-Sn}.   
Then \[ \rk(\mathrm{Fix}(\wedge^2 \pi)) \le \binom{n - 1}{2}. \] with equality holding whenever $\pi$ a transposition. 

\end{lem}
\begin{proof}
It is seen without difficulty that the submodule $\mathrm{Fix}(\wedge^2 \pi)$ is generated by the sums of elements in the orbits of the cyclic subgroup $C_\pi$ generated by $\pi$ acting on the set $\bar B$ as defined in the discussion preceding the theorem. 
Needless to say such ``orbit sums" need not be linearly independent (over $\mathbb Z$) or even be non-zero. For example the $C_{(ij)}$-orbit $\{\mathbf e_i \wedge \mathbf e_j, -\mathbf e_i \wedge \mathbf e_j\}$ has zero sum.  
Let  $\mathcal N_\pi$ the number of $C_{(ij)}$-orbits in this action on the set $\bar B$.
By an application of the well-known formula of Burnside it can be shown (\cite{Sil}) that the number $N_\pi$ is maximal when $\pi$ is a transposition $(ij)$ and in this case 
\[ \mathcal N_{(ij)} = (n - 2)(n - 3) + (2n - 3). \] 
It is also clear that the orbit $C_{(ij)}(\mathbf e_{i'} \wedge \mathbf e_{j'})$ has a nonzero sum whenever $(i,j) \ne (i',j')$ and that this sum is the negative of the sum of the orbit $C_{(ij)}(-\mathbf e_{i'} \wedge \mathbf e_{j'})$.  
We thus obtain \[ \rk \bigl (\mathrm{Fix}(\wedge^2 (ij)) \bigr ) =  \frac{\mathcal N_{(ij)} - 1}{2} = \binom{n - 1}{2}\] $\mathbb Z$-independent orbit-sums that constitute a $\mathbb Z$-basis for $\mathrm{Fix}(\wedge^2 \pi)$.
In general the group $C_\pi$ will have some orbits whose sum is zero. Clearly these orbits are precisely the orbits $C_\pi(\mathbf e_i \wedge \mathbf e_j)$ such that $C_\pi( \pm \mathbf e_i \wedge \mathbf e_j) \ni  \mp \mathbf e_i \wedge \mathbf e_j$ and the orbits $C_{\pi}$ with non-zero sum are precisely those for which \[ C_{\pi}(\mathbf e_i \wedge \mathbf e_j) \cap C_{\pi}(- \mathbf e_i \wedge \mathbf e_j) = \emptyset. \] 
It is immediate from this that the number of independent orbit-sums for the permutation $\pi$ are bounded above by $\floor*{\frac{\mathcal N_\pi}{2}}$. 

Evidently for any non-identity permutation $\pi \in S_n$ the number of fixed points $\mathrm{Fix}(\pi)$ is bounded above by $n - 2$ and this bound in attained only by a transposition $(ij)$. Using this fact and the Burnside formula an upper bound for $\mathcal N_\pi$ was obtained in \cite{Sil} as follows  
\begin{align*} \mathcal N_\pi &=  \frac{1}{\abs {C_\pi}}\Biggl ( n(n - 1) + \sum_{\phi \in C_\pi , 
\phi \ne 1}2\binom{\mathrm{Fix}(\phi)}{2} \Biggr ) \le \frac{ n(n - 1)}{\abs{C_\pi}} +   \frac{(\abs{C_\pi} - 1)}{\abs{C_\pi}}2\binom{n - 2}{2} \\ &= 
(n - 2)(n -3) + \frac{4n -6}{\abs{C_\pi}}. 
\end{align*} 
This  yields \[ \mathcal N_{(ij)} - \mathcal N(\pi) \ge (4n - 6)\Biggl (\frac{1}{2} - \frac{1}{\abs{C_\pi}} \Biggr )  
 \] 
and thus for $n \ge 2$
\begin{align*} \rk(\mathrm{Fix}\bigl (\wedge^2 (ij)) \bigr ) - \rk(\mathrm{Fix}\bigl (\wedge^2 \pi) \bigr ) &= \frac{\mathcal N_{(ij)} - 1}{2} - \floor*{\frac{\mathcal N_\pi}{2}} \\ &\ge -\frac{1}{2} + \frac{\mathcal N_{(ij)} - \mathcal N_\pi}{2} \ge -\frac{1}{2} + (2n - 3)\Biggl (\frac{1}{2} - \frac{1}{\abs{C_\pi}}\Biggr )   \ge -\frac{1}{2}.  \end{align*}
As $\frac{\mathcal N_{(ij)} - 1}{2} , \floor*{\frac{\mathcal N_\pi}{2}} \in \mathbb Z$ it clearly follows that \[ \rk(\mathrm{Fix}\bigl (\wedge^2 (ij)) \bigr ) \ge \rk(\mathrm{Fix}\bigl (\wedge^2 \pi) \bigr ). \]
\end{proof}
\begin{thm_B}
 A quantum affine space  $\mathcal O_{\mathfrak q}(\mathbb F^n)$ such that the subgroup $\Lambda$ has rank no smaller than $\binom{n - 1}{2} + 1$ satisfies 
 \[ \aut(\mathcal O_{\mathfrak q}(\mathbb F^n)) = (\mathbb F^\ast)^n. \]  

\end{thm_B}

\begin{proof}
As usual we write 
$\mathcal O_{\mathfrak q} = \mathcal O_{\mathfrak q}(\mathbb F^n)$.
The assumption concerning the rank of the group $\Lambda$ means that the corresponding quantum torus $\widehat {\mathcal O}_{\mathfrak q}$ (arising from localization) of the given quantum affine space $\mathcal O_{\mathfrak q}(\mathbb F^n)$ has center $\mathcal Z$ no bigger than the ground field $\mathbb F$. Indeed viewing the latter algebra as a crossed product $\mathbb F \ast \Gamma$ ($\Gamma = \mathbb Z^n$) as in \cite{OP1995} it is easily seen that for any subgroup $\Gamma'$ of $\Gamma$ of finite index the corresponding $\lambda$-group $\Lambda'$ of the sub-quantum torus $\mathbb F \ast \Gamma'$ has the same rank as that of the group $\Lambda$. 

It is well known (e.g., \cite{OP1995}) that the center of a quantum torus is a generated by monomials.  
Given a central monomial $\mathbf  z: = \mathbf X^{\mathbf m}$ in the algebra $\widehat {\mathcal O}_{\mathfrak q}$ we can clearly extend the set $\{ \mathbf z \}$ to a set of $n$ monomials which together with their inverses generate a subalgebra of the form $\mathbb F \ast \Gamma'$ as described in the previous paragraph.
But the central monomial $\mathbf z$ evidently reduces $\rk(\Lambda')$ by $n - 1$ implying $\rk(\Lambda) \le \binom{n - 1}{2}$ and thus contradicting the theorem hypothesis. We also note that by the theorem of Section 1.3 of \cite{MP}  the algebra $\mathcal O_{\mathfrak q}$ is simple. 

Let $\pi$ be a permutation of the generators $X_1, \cdots X_n$ of the generators of the algebra $\widehat {\mathcal O}_{\mathfrak q}$. 
To show the assertion in the theorem it suffices in view of \cite[Propositon 3.2]{OP1995} to show that 
if $\pi \in \aut(\mathcal O_{\mathfrak q})$ then $\pi = \mathrm{id}$. To this end we note that by \cite[Proposition 1.5]{OP1995} the permutation $\pi$ extends to an automorphism of the quantum torus $\widehat {\mathcal O}_{\mathfrak q}$. Moreover this extension induces on $\Gamma$ the automorphism which is given by the permutation matrix $P$ corresponding to $\pi$. By Remark \ref{key-rmk} \[ \mathrm{Fix}\wedge^2 \pi \ge \binom{n - 1}{2} + 1 \]    
contradicting Lemma \ref{key-lem}     
\end{proof}

\subsection{Examples} 

\begin{example} In the situation of Example 2 of Section \ref{aut-qtor} we consider the quantum affine space defined by the same matrix $\mathfrak q$ of multiparameters. In this case the $\lambda$-group has rank $2$ and so Theorem B applies. In particular this quantum space is a simple ring and whose automorphisms (permutations of generators) lift to automorphisms of the corresponding quantum torus $\widehat{\mathcal O}_{\mathfrak q}$ . As seen in the same example referred to above these automorphisms induce on $\Gamma$ automorphisms having the form $\begin{pmatrix}
\epsilon & 0 & a \\
0 & \epsilon & b \\
0 & 0 & \epsilon
\end{pmatrix}$ 
But the only permutation matrix which has this form is the identity. In other words the assertion of Theorem B holds true in this case.   
\end{example}

\begin{example}\label{5_circ}
Consider the multiplicatively antisymmetric matrices $\mathfrak q \in M_5(\mathbb F)$ defined by the $5 \times 5$ grid of Figure \ref{fig:M1}
\begin{figure}[htbp]
\centering
\begin{tikzpicture}[every node/.style={minimum size=.5cm-\pgflinewidth, outer sep=0pt}]
    \draw[step=0.5cm,color=black] (0,0) grid (2.5,2.5);
    \node[] at (-0.75,+0.75) {};
    \node[pattern=custom north west lines,hatchspread=3pt,hatchcolor=gray] at (+0.75,+2.25) {};
    \node[pattern=custom north west lines,hatchspread=3pt,hatchcolor=gray] at (+1.25,+1.75) {};
    \node[pattern=custom north west lines,hatchspread=3pt,hatchcolor=gray] at (+1.75,+1.25) {};
    \node[pattern=custom north west lines,hatchspread=3pt,hatchcolor=gray] at (+2.25,+0.75) {};
    \node[pattern=custom north west lines,hatchspread=3pt,hatchcolor=gray] at (+0.25,+0.25) {};
    \node[pattern= north east lines,hatchspread=4pt,hatchcolor=gray] at (+0.25,+1.75) {};
    \node[pattern=north east lines,hatchspread=3pt,hatchcolor=gray] at (+0.75,+1.25) {};
    \node[pattern=north east lines,hatchspread=3pt,hatchcolor=gray] at (+1.25,+0.75) {};
    \node[pattern=north east lines,hatchspread=3pt,hatchcolor=gray] at (+1.75,+0.25) {};
    \node[pattern= north east lines,hatchspread=3pt,hatchcolor=gray] at (+2.25,+2.25) {};
    \node[pattern=horizontal lines,hatchspread=3pt,hatchcolor=gray] at (+1.25,+2.25) {};
    \node[pattern=horizontal lines,hatchspread=3pt,hatchcolor=gray] at (+1.75,+1.75) {};
    \node[pattern=horizontal lines,hatchspread=3pt,hatchcolor=gray] at (+2.25,+1.25) {};
    \node[pattern=horizontal lines,hatchspread=3pt,hatchcolor=gray] at (+0.25,+0.75) {};
    \node[pattern=horizontal lines,hatchspread=3pt,hatchcolor=gray] at (+0.75,+0.25) {};     
    \node[pattern=horizontal lines,hatchspread=3pt,hatchcolor=gray] at (+1.25,+2.25) {};
    \node[pattern=horizontal lines,hatchspread=3pt,hatchcolor=gray] at (+1.75,+1.75) {};
    \node[pattern=horizontal lines,hatchspread=3pt,hatchcolor=gray] at (+2.25,+1.25) {};
    \node[pattern=horizontal lines,hatchspread=3pt,hatchcolor=gray] at (+0.25,+0.75) {};
    \node[pattern=horizontal lines,hatchspread=3pt,hatchcolor=gray] at (+0.75,+0.25) {};
    \node[pattern=vertical lines,hatchspread=3pt,hatchcolor=gray] at (+1.75,+2.25) {};
    \node[pattern=vertical lines,hatchspread=3pt,hatchcolor=gray] at (+2.25,+1.75) {};
    \node[pattern=horizontal lines,hatchspread=3pt,hatchcolor=gray] at (+2.25,+1.25) {};
    \node[pattern=vertical lines,hatchspread=3pt,hatchcolor=gray] at (+0.25,+1.25) {};
    \node[pattern=vertical lines,hatchspread=3pt,hatchcolor=gray] at (+1.25,+0.25) {};
    \node[pattern=vertical lines,hatchspread=3pt,hatchcolor=gray]  at (+0.75,+0.75) {};
\end{tikzpicture}
\caption{Multiplicatively antisymmetric matrix $\mathfrak q$ with circulant symmetry} \label{fig:M1}
\end{figure}
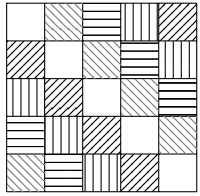
where the boxes with the same pattern indicate that the corresponding entries of $\mathfrak q$ are equal. Furthermore, the (equal) entries in the boxes with northeast hatching are inverses of the (equal) entries in the boxes with northwest hatching. 
Similarly for the boxes with vertical and horizontal hatching. The unhatched boxes correspond to $1 \in \mathbb F$. 

Thus the matrices arising in this way are multiplicatively antisymmetric as well as circulant. Evidently, the rank of the $\lambda$-group in this case is  at most $2$.  

Using \eqref{OP-criterion2} it is easily seen that the the 5-cycle $(12345)$ is an automorphism of $\mathcal O_{\mathfrak q}$ which clearly lifts to an automorphism of the corresponding quantum torus $\widehat{\mathcal O}_{\mathfrak q}$ (and thus induces an automorphism of $\mathbb Z^5$). As $\rk (\mathrm{Fix}(\wedge^2 (12345)) = 2$ this is agreement with Remark \ref{key-rmk}.

\end{example}

\newpage



\end{document}